\documentclass[reqno,12pt]{amsart}
\usepackage{amsfonts,amssymb}
\theoremstyle{plain}
\newtheorem{theorem}{Theorem}
\newtheorem{lemma}{Lemma}
\newtheorem{corollary}{Corollary}
\newtheorem{proposition}{Proposition}

\theoremstyle{definition}
\newtheorem{definition}{Definition}
\theoremstyle{remark}
\newtheorem{remark}{Remark}
\newtheorem{example}{Example}

\begin{document}

\title[The wavelet transforms in Gelfand-Shilov spaces]
{The wavelet transforms in Gelfand-Shilov spaces}

\author[S. Pilipovi\'{c}]{Stevan Pilipovi\'{c}}
\address{S. Pilipovi\'{c}\\ University of Novi Sad, Faculty of Sciences\\ Department
 of Mathematics and Informatics\\ Trg Dositeja Obradovi\' ca 4\\ 21000 Novi Sad \\ Serbia }
\email{pilipovics@yahoo.com}

\author[D. Raki\'{c}]{Du\v{s}an Raki\'{c}}
\address{D. Raki\'{c}\\ University of Novi Sad\\ Faculty of Technology\\ Bul. cara Lazara 1\\ 21000 Novi Sad \\ Serbia}
\email{drakic@tf.uns.ac.rs}

\author[N. Teofanov]{Nenad Teofanov}
\address{N. Teofanov\\ University of Novi Sad, Faculty of Sciences\\ Department
 of Mathematics and Informatics\\ Trg Dositeja Obradovi\' ca 4\\ 21000 Novi Sad \\ Serbia }
\email{nenad.teofanov@dmi.uns.ac.rs}

\author[J. Vindas]{Jasson Vindas}
\address{J. Vindas\\ Department of Mathematics\\ Ghent University\\ Krijgslaan 281 Gebouw S22\\ B 9000 Gent\\ Belgium}
\email{jvindas@cage.UGent.be}


\begin{abstract}
We describe local and global properties of wavelet transforms of
ultra-differentiable functions. The results are given in the form of continuity properties of the wavelet transform on Gelfand-Shilov type spaces
and their dual spaces. In particular, we introduce
a new family of highly time-scale localized spaces on the upper half-space. We study
the wavelet synthesis operator (the left-inverse of the wavelet transform) and obtain the
resolution of identity (Calder\'{o}n reproducing formula) in the context of ultradistributions.
\end{abstract}

\maketitle

\section{Introduction}

One of the most useful concepts in time-frequency analysis for signal analysts and engineers
is the wavelet series expansion of a signal.
The coefficients in such series, representing the discrete version of a signal,
are then used in the signal analysis, processing and synthesis.
The continuous versions
of these discrete representations lead to the wavelet (analysis) transform
$ \mathcal{W}_{\psi}$ and the wavelet synthesis operator $\mathcal{M}_{\phi}$ \cite{hol1}.
The authors have
studied both transforms in several papers, \cite{PRV,vindas-pilipovic,vindas-pilipovic-2,RT,VPR}.
Although the continuous transforms are less popular in the literature than their
discrete counterparts, studying the intrinsic properties of the continuous wavelet transform is also
a very important subject. In particular, continuous transforms may potentially serve well in the study
of microlocal and pointwise aspects of a signal, cf. \cite{FFV,hol99,jaffard-m,PVul}.
Microlocal aspects have also been recently studied by different authors via shearlet transforms, see e.g.
\cite{DK,KL}. An interesting alternative approach to the wavelet transform
in several variables with applications in microlocal analysis is performed in \cite{FFV}.

It is well known that smooth orthonormal wavelets cannot have
exponential decay, cf. \cite{Dau,DzH,HW}. In this paper we study the
wavelet transform defined by wavelets with almost exponential decay.
In this context it is then natural to work with Gelfand-Shilov spaces
as a functional-analytic groundwork. We shall prove continuity theorems
for the wavelet transform and the wavelet synthesis operator on spaces of
Gelfand-Shilov type, see Section \ref{Sec1} for definitions. In
contrast to known results \cite{hol1,pa06,PRV,RT}, we introduce in the article
a new family of (semi-)norms with additional parameters in the corresponding wavelet image space. These parameters measure fast decay or growth orders of the wavelet
transform and the wavelet synthesis operator. Roughly speaking, our considerations are able to detect Gevrey ultra-differentiability properties (such as analyticity) via appropriate decay of the wavelet transform.

\par

Gelfand-Shilov spaces of ultra-differentiable functions
were originally introduced in \cite{GS} as a tool to treat existence and uniqueness questions
for parabolic initial-value problems. Such spaces, consisting of Gevrey
ultra-differentiable functions, are also very useful in hyperbolic and
weak hyperbolic problems, see \cite{BL,GaRu,R} and the references therein.
Exponential decay and holomorphic extension of solutions to globally elliptic equations
in terms of Gelfand-Shilov spaces have been recently studied in \cite{CGR-1,CGR-2}, see also \cite{BL}.
We refer to \cite{NR} for an overview of results in this direction
and for applications in quantum mechanics and traveling waves.
On the other hand, in the context of time-frequency analysis, the Gelfand-Shilov spaces
have recently captured much attention in connection with modulation spaces
\cite{GZ}, localization operators \cite{CPRT},
and the corresponding pseudodifferential calculus \cite{Prangoski,Toft-1,Toft-2}.
We follow here Komatsu's approach \cite{K} to spaces of
ultra-differentiable functions. Another widely used approach is that of Braun, Meise, Taylor, Vogt and their
collaborators, see e.g. \cite{BMT} and the recent contribution \cite{RS}.
These two approaches are equivalent in many interesting situations,
cf. \cite{LO} for more details.

\par

We remark that the wavelet transform in the context of
Gelfand-Shilov spaces was already studied in \cite{pa06,pa09} in
dimension $n=1$. In the present article we propose and develop an
intrinsically different approach, which also covers the
multidimensional case. We employ here wavelets with all vanishing
moments. The advantage of this condition is that one is able to
translate ultra-differentiability and subexponential decay of
functions into sharper localization properties in the scale variable
of the wavelet transform. Our approach also provides  the resolution of the identity (Calder\'{o}n
reproducing formula) for ultradistributions. As a matter of fact,
this inversion formula for the wavelet transform of
ultradistributions seems to be out of reach of the results from
\cite{pa06,pa09}.

We point out that the number of vanishing moments (called {\em
cancellations} in \cite{M}) of a wavelet $\psi$ is intimately
related to {\em the order of approximation} of the corresponding
wavelet series via the so-called Strang-Fix condition. In
particular, wavelets with many vanishing moments are appropriate
when dealing with objects which are very regular except for a few
isolated singularities, cf. \cite{Dau,Mallat}. It is also well known
\cite{HW} that as soon as an orthogonal wavelet belongs to the
Schwartz class $\mathcal{S}(\mathbb{R}^{n})$ then all its moments
must vanish. In \cite{hol1} wavelets with all vanishing moments were
used to develop a distributional framework for the wavelet transform
in the context of Lizorkin spaces. Here we shall develop an new ultradistributional framework.

\par

The paper is organized as follows. In Section \ref{Sec1} we explain some facts about Gelfand-Shilov type spaces. In particular, we introduce a new four-parameter family $
\mathcal{S}^{s}_{t, \tau_1, \tau_2} (\mathbb{H}^{n+1}) $ of function spaces on the upper-half space and study its properties (see Subsection \ref{G-S upper half-space}). Section \ref{Sec2} contains our main continuity results, Theorems \ref{th3} and \ref{th4}, which imply Calder\'{o}n reproducing formulas for ultradistributions (Theorem \ref{th-res-identity} and Corollary \ref{desingularization ultra}). Finally, in Section \ref{Sec3} we collect the proofs of the main results.

\subsection{Notation and notions}

We denote by $\mathbb{H}^{n + 1} = \mathbb{R}^n \times \mathbb{R}_+ $ the upper half-space
and $\mathbb{N} = \{ 0,1,2, \dots \}$. The unit sphere in $ \mathbb{R}^n$ is denoted by $\mathbb{S}^{n-1}$.
When $ x, y \in \mathbb{R}^n $ and $ m \in
\mathbb{N}^n $, $|x|$ denotes the Euclidean norm, $ \langle x \rangle = (1+|x|^2)^{1/2}$,
$ xy = x_1 y_1 + x_2 y_2 + \dots + x_n y_n, $ $ x^{m} = x_1^{m_1} \dots x_n^{m_n},$
$m! = m_1! m_2! \dots m_n! $,
$\partial^{m}=\partial_x^{m} =
\partial_{x_1}^{m_1} \dots\partial_{x_n}^{m_n},$ and $ \triangle $ denotes the Laplacian: $
\triangle = \triangle_x = \partial_{x_1 ^2 }^{2} + \dots + \partial_{x_n ^2}^{2}.$
By a slight abuse of notation, the length of a multi-index
$ m \in \mathbb{N}^n $ is denoted by $ |m| = m_1 + \dots + m_n $, and
the meaning of $|\cdot| $ shall be clear from the context.
We denote by $C,  h, \dots$  positive constants which may be different in various occurrences;
$ A\lesssim B $ means that $ A \leq C \cdot B $ for
some positive constant $C$.
If $ A\lesssim B $  and  $ B\lesssim A $ we write $ A \asymp B $.
The dual pairing between a test function space $ {\mathcal A}$ and its dual space of (ultra)distributions
${\mathcal A'}$ is denoted by
$\langle\: \cdot\:, \:\cdot\: \rangle=_{\mathcal A'}\langle\: \cdot\:, \:\cdot \:\rangle_{\mathcal A} $.

When $ \alpha$ and $ \beta $ are multi-indices and $n$ is the space dimension,
we have
$$
|\alpha|! \leq n^{|\alpha|}\alpha !, \;\;\;
 \alpha! \beta! \leq  (\alpha +\beta)! \leq 2^{|\alpha|+|\beta|} \alpha ! \beta!.
$$

\par

\section{Gelfand-Shilov type spaces}  \label{Sec1}

In this section we discuss definitions and properties of the test function spaces
that will
be employed in our study of the wavelet transform.

\par

For the reader's convenience, and in order to be self-contained,
we first recall various spaces of rapidly decreasing functions that were considered in the context of wavelet transform
in e.g. \cite{hol1,PRV}.

\par

The moments of $\varphi\in\mathcal{S}(\mathbb{R}^{n})$,
the Schwartz space of rapidly decreasing smooth test functions,
are denoted by $\mu_{m}(\varphi)=\int_{\mathbb{R}^{n}} x^{m}\varphi(x)dx$,
$m\in\mathbb{N}^{n}$. We fix constants in the Fourier
transform as follows: $\hat{\varphi}(\xi) = \int_{\mathbb{R}^{n}}
\varphi(x) e^{-ix \cdot \xi} \mathrm{d}x,$ $\xi\in\mathbb{R}^{n}$.

\par

The Lizorkin space  $\mathcal{S}_0 (\mathbb{R}^{n}) = \left\{\varphi\in
\mathcal{S}(\mathbb{R}^{n}): \: \mu_m (\varphi) = 0,\ \forall m
\in \mathbb{N}^{n} \right\} $ is a closed subspace of
$\mathcal{S}(\mathbb{R}^{n})$ equipped with the relative
topology inhered from $\mathcal{S}(\mathbb{R}^{n})$, \cite{hol1,Samko}.

\par

The space $ \mathcal{S} (\mathbb {H}^{n+1}) $ of ``highly localized
functions over the half-space'' \cite{hol1} consists of  $ \Phi
\in C^{\infty} (\mathbb{H}^{n+1}) $ such that the seminorms
\begin{equation} \label{norma-hl}
p^{l,k} _{\alpha, \beta}(\Phi)=\sup_{(b,a)\in
\mathbb {H}^{n+1}}\left(a^{l}+\frac
{1}{a^{l}}\right)
\langle b \rangle ^k \,
\left| \partial^{\alpha} _{a} \partial^{\beta} _{b} \Phi (b,a)\right |
\end{equation}
are finite for all $l,k,\alpha \in \mathbb{N}$ and for all
$ \beta \in\mathbb{N}^{n}$.

\par

When $ (b,a) \in \mathbb {H}^{n+1}$ and $ k,l\in  \mathbb{N}$, then $ \left(a^{l}+\frac
{1}{a^{l}}\right) \asymp \left(a +\frac
{1}{a}\right)^{l} $ and
$ \langle b \rangle ^k \asymp | b | ^k $ when $|b|$ is large enough, see also
Remark \ref{remark equivalent sequences} below.

\par

We introduce spaces of Gelfand-Shilov type by analogy to
$\mathcal{S} (\mathbb{R}^{n})$,  $\mathcal{S}_0 (\mathbb{R}^{n})$ and
$ \mathcal{S} (\mathbb {H}^{n+1}) $. The family of spaces $ \mathcal{S}_{{\rho_2}}^{{\rho_1}} (\mathbb{R}^n)$
introduced by I. M. Gelfand and G. E. Shilov
in the study of Cauchy problems was systematically studied in
\cite{GS}, see \cite{NR} for a recent survey.

\par

Recall that $ \varphi \in
\mathcal{S}(\mathbb{R}^n) $ belongs to the  Gelfand-Shilov space $
\mathcal{S}_{{\rho_2}}^{{\rho_1}} (\mathbb{R}^n),$ $ {\rho_1}, {\rho_2} > 0,$ if  there exists  a constant $  h > 0 $
such that
$$
|x^{\alpha} \varphi^{(\beta)} (x)| \lesssim h^{-|\alpha + \beta|}
\, \alpha!^{{\rho_2}} \beta!^{{\rho_1}}, \; \; x \in \mathbb{R}^n, \,
\alpha, \beta \in \mathbb{N} ^n.
$$
The space $ \mathcal{S}_{{\rho_2}}^{{\rho_1}} (\mathbb{R}^n)$ is nontrivial if and only if $ {\rho_1} + {\rho_2} \geq 1 $. The family of
norms
\begin{equation} \label{norma-GS}
 p_h^{{\rho_1}, {\rho_2}} (\varphi) = \sup_{x \in \mathbb{R}^n, \alpha,
\beta \in \mathbb{N} ^n} \frac{h^{|\alpha + \beta|}}{\alpha!^{{\rho_2}}
\beta!^{{\rho_1}}} \, | x^{\alpha}\varphi^{(\beta)} (x)|, \;\: h > 0,
\end{equation}
defines the canonical inductive topology of $ \mathcal{S}_{{\rho_2}}^{{\rho_1}} (\mathbb{R}^n). $

It is well known \cite{chchki96} that $\varphi \in \mathcal{S}_{{\rho_2} } ^{{\rho_1}}
(\mathbb{R}^n) $ if and only if there exists $h>0$ such
that
$$
\sup_{x\in\mathbb{R}^{n}}  e^{h|x|^{1/{\rho_2}}} | \varphi (x)| <\infty \;\;\; \mbox{and}
\;\;\; \sup_{\xi\in\mathbb{R}^{n}}    e^{h|\xi|^{1/{\rho_1}}} | \hat \varphi (\xi)|<\infty.
$$
Hence, the Fourier transform is an isomorphism between
$ \mathcal{S}_{{\rho_2}}^{{\rho_1}} (\mathbb{R}^n) $
and $ \mathcal{S} ^{{\rho_2}} _{{\rho_1}} (\mathbb{R}^n). $

\par

The space $ \mathcal{S}_{{\rho_2}}^{{\rho_1}} (\mathbb{R}^n) $ is non-quasianalytic, namely, it contains compactly supported functions, if and only if $\rho_1>1$.
Then it consists of Gevrey ultra-differentiable functions, cf. \cite{R}.
If $\rho_1=1$, then $ \varphi \in \mathcal{S}_{{\rho_2}}^{{\rho_1}} (\mathbb{R}^n) $ is
a real analytic function, and if $0<\rho_1<1$,
then it is an entire function.

\par

\begin{remark} \label{remark equivalent sequences}
We will often use an equivalent family of norms where in
\eqref{norma-GS} (and in other similar situations)
$ x^{\alpha}\varphi^{(\beta)} (x)$
is replaced by $ \langle x \rangle^{\alpha}\varphi^{(\beta)} (x)$,
 $ (\langle x \rangle^{\alpha}\varphi (x))^{(\beta)} $
or $ (x^{\alpha}\varphi (x))^{(\beta)} $. Moreover, instead of the supremum norm any $L^p$ norm ($1\leq p<\infty$)
gives rise to an equivalent topology on $ \mathcal{S}_{{\rho_2}}^{{\rho_1}} (\mathbb{R}^n) $
(cf. \cite[Ch 2.5]{CKP}).
\end{remark}

\par

We denote by $ (\mathcal{S}_{{\rho_2} } ^{{\rho_1}})_0 (\mathbb{R}^n) $ the closed subspace of
$ \mathcal{S}_{{\rho_2}}^{{\rho_1}} (\mathbb{R}^n)$ given by
$$
(\mathcal{S}_{{\rho_2} }^{{\rho_1}})_0
(\mathbb{R}^n)=\left\{\varphi \in \mathcal{S}_{{\rho_2}}^{{\rho_1}}
(\mathbb{R}^n): \: {\mu}_m (\varphi) = 0,\ \forall m \in
\mathbb{N}^{n} \right\}.
$$

\par

One can show that $ (\mathcal{S}_{{\rho_2}}^{{\rho_1}})_0 (\mathbb{R}^n) $, $ {\rho_1}, {\rho_2} > 0,$ is
nontrivial if and only if $ {\rho_2} > 1 $ (cf. \cite{GS}).

\par

\subsection{Gelfand-Shilov type spaces on the upper half-space}\label{G-S upper half-space}
In this subsection we introduce a new scale of function spaces which describes
sharp subexponential/superexponential localization over the upper half-space.

To this end, we employ parameters which measure the decay properties of a function
with respect to the scaling variable $a>0$ at zero and at infinity, as well
as their Gevrey ultra-differentiability and decay properties in the localization variable $b.$
While the seminorms in \eqref{norma-hl} measure polynomial decay of a certain order with
respect to the scaling parameter $a>0$ at zero and at infinity,
the seminorms in  \eqref{norma-GSupper}
may detect (super- and sub-) exponential decay of different orders at zero and at infinity.

\par

\begin{definition} \label{upper-half-space}
Let $ s, t, \tau_1, \tau_2 > 0 $.
A smooth function $ \Phi $ belongs to $ \mathcal{S}^{s}_{t, \tau_1, \tau_2}
(\mathbb{H}^{n+1}), $  if for every $\alpha \in \mathbb{N} $
there exists a constant $h>0$
such that
\begin{equation} \label{norma-GSupper}
p_{\alpha, h}^{s, t, \tau_1, \tau_2} (\Phi) = \sup
\frac{h^{|\beta| + k + l_1 + l_2}}{\beta!^{s} k!^{t} l_1!^{\tau_1}
l_2!^{\tau_2}} \, \left(a^{l_1} + \frac{1}{a^{l_2}}\right) \,
\langle b \rangle^k\left|  \, \partial_a^{\alpha}
\partial_b^{\beta} \Phi (b, a) \right| < \infty,
\end{equation}
where the supremum is taken over
$$ ((b, a), (k, l_1, l_2), \beta) \in \Lambda = \mathbb{H}^{n+1} \times \mathbb{N}^3 \times
\mathbb{N}^n.
$$
\end{definition}

\par

The topology of $ \mathcal{S}^{s}_{t,\tau_1, \tau_2}
(\mathbb{H}^{n+1}) $ is defined via the family of seminorms
\eqref{norma-GSupper}, as inductive limit with respect to $h$ and projective limit with respect to $\alpha$.

The space $ \mathcal{S}^{s}_{t,\tau_1, \tau_2} (\mathbb{H}^{n+1}) $ is
nontrivial if and only if $ s+ t \geq 1$, which can be proved as follows.

Consider the set of smooth functions in $\mathbb H^{n+1}$ of the form
$\Phi(b,a)=g(b)f(a)$, $b\in\mathbb R^n$, $a\in\mathbb R_+.$
Then  $ p_{\alpha, h}^{s, t, \tau_1, \tau_2} (\Phi)<\infty $ is equivalent to
$ p^{s,t}_h(g)<\infty $ and
\begin{equation}\label{dod1}
\sup_{a>0,l_1,l_2\in \mathbb N}
\frac{h^{l_1 + l_2}}{ l_1!^{\tau_1}
l_2!^{\tau_2}} \, \left(a^{l_1} + \frac{1}{a^{l_2}}\right) |\partial_a^{\alpha}
 f(a)|< \infty.
\end{equation}
Thus, if $s+t\geq 1,$ then $ \mathcal{S}^{s}_{t,\tau_1, \tau_2} (\mathbb{H}^{n+1}) $ is non-trivial,
$\tau_1, \tau_2 > 0 $.
For example, if $g \in\mathcal S^s_t(\mathbb R^n),$ then
$ \mathbb H^{n+1}\ni(b,a)\mapsto e^{-a^{1/\tau_1} -a^{-1/\tau_2}} g(b) \in
\mathcal{S}^{s}_{t,\tau_1, \tau_2} (\mathbb{H}^{n+1}) $.

Since for fixed $a \in\mathbb R_+ $ and $ l_1 = l_2 = \alpha = 0 $,
it follows from \eqref{norma-GSupper} that $  \Phi (\cdot , a) \in \mathcal{S}^s_t(\mathbb R^n),$
we see that the condition
$s+t\geq 1$ is also necessary for the non-triviality of
$ \mathcal{S}^{s}_{t,\tau_1, \tau_2} (\mathbb{H}^{n+1}). $

\par

Obviously, the family $ \mathcal{S}^{s}_{t,\tau_1, \tau_2} (\mathbb{H}^{n+1}) $ is increasing with
respect to parameters $ s,t, \tau_1, \tau_2 $. The parameters $ \tau_1 $ and $\tau_2 $
measure the behavior of $ \Phi \in \mathcal{S}^{s}_{t, \tau_1, \tau_2} (\mathbb{H}^{n+1}), $
with respect to $a>0$ at infinity and at zero, respectively.

\par

It can be shown that all these spaces of test functions
are closed under multiplication by (ultra-)polynomials, partial differentiation (or more generally ultra-differential operators),
translation and  dilation, cf. \cite{GS} for $ \mathcal{S}_{{\rho_2}}^{{\rho_1}}
(\mathbb{R}^n) $.
The following lemma can be proved in the same way as it is done in
\cite[Chapter IV 6.2]{GS} for $ \mathcal{S}_{{\rho_2}}^{{\rho_1}} (\mathbb{R}^n)$, we therefore omit its proof.

\begin{lemma} \label{FTH}
Let  $ \Phi  \in C^\infty  (\mathbb{H}^{n+1}) $ and let
$ \mathcal{F}_1 \Phi $ denote its Fourier transform  with respect to the first variable:
$$
\mathcal{F}_1 \Phi (\xi, a) = \int_{\mathbb{R}^n} e^{- i b
\xi} \, \Phi (b, a) \, db, \; \; (\xi, a) \in \mathbb{H}^{n+1}.
$$
Then $ \Phi  \in {\mathcal S}^{s}_{t, \tau_1, \tau_2} (\mathbb{H}^{n+1}) $ if and only if
$ \mathcal{F}_1 \Phi \in  {\mathcal S}^{t}_{s, \tau_1, \tau_2} (\mathbb{H}^{n+1}).$
Furthermore, $ \mathcal{F}_1 $ is a topological isomorphism between
$ {\mathcal S}^{s}_{t, \tau_1, \tau_2} (\mathbb{H}^{n+1}) $ and
$ {\mathcal S}^{t}_{s, \tau_1, \tau_2} (\mathbb{H}^{n+1}). $
\end{lemma}

\par

Next, we show that \eqref{norma-GSupper} precisely describes the rate of decay of the derivatives of $ \Phi $.

\begin{proposition} \label{prop2}
Let $ \Phi  \in {\mathcal S}^{s}_{t, \tau_1, \tau_2} (\mathbb{H}^{n+1}) $
and  $ \alpha \in \mathbb{N}$. Set
$$
q_{\alpha, h}^{s, t, \tau_1, \tau_2} (\Phi) := \sup_{((b, a), \beta) \in
\mathbb{H}^{n+1} \times \mathbb{N}^n}
\frac{h^{|\beta|}}{\beta!^{s}} \, e^{h\left(a^{1/\tau_1} +
a^{-1/\tau_2} + |b|^{1/t} \right)} \, \left|\partial_a^{\alpha}
\partial_b^{\beta} \Phi (b, a) \right|.
$$
Then
$ p_{\alpha, h}^{s, t, \tau_1, \tau_2} (\Phi) < \infty $ for some $h>0$,
if and only if $ q_{\alpha, h}^{s, t, \tau_1, \tau_2}(\Phi) < \infty $ for some  $h>0$.
\end{proposition}

\begin{proof}
Assume that $ p_{\alpha, h}^{s, t, \tau_1, \tau_2} (\Phi) < \infty $ for some $h>0$.
Then, for any given $ l_1, l_2, k \in \mathbb{N} $,
$$
\frac{h^{l_1 + l_2 + k}}{l_1!^{\tau_1} l_2!^{\tau_2} k!^{t}}
\left(a^{l_1} + \frac{1}{a^{l_2}}\right) \langle b \rangle^{k}
| \partial_a^{\alpha} \Phi(b,a)|
$$
is uniformly bounded on $\mathbb{H}^{n+1}.$
This implies that appropriate summations over
$ l_1, l_2$ and $ k $ are also uniformly bounded.
Indeed, the estimate
$$
C^{-1} \, e^{(r - \varepsilon) \eta^{1/r}}
\leq \sum_{j = 0}^{\infty} \frac{\eta^j}{(j!)^r} \leq C \, e^{(r + \varepsilon)
\eta^{1/r}}, \; \forall \eta \geq 0,
$$
which holds for every $ r, \varepsilon > 0 $ and for some $ C = C(r, \varepsilon) > 0$, yields
$$
|\partial_a^{\alpha} \partial_b^{\beta} \Phi(b,a)| \lesssim \tilde{h}^{|\beta|} \beta!^s e^{-
\tilde{h} \left(a^{\frac{1}{\tau_1}} +
(\frac{1}{a})^{\frac{1}{\tau_2}} + |b|^{\frac{1}{t}}\right)},
\;\;\; (b, a) \in \mathbb{H}^{n+1}, \beta \in \mathbb{N}^n,
$$
for some $ \tilde{h} > 0 $. By taking the corresponding supremum, we obtain that
$ q_{\alpha, h}^{s, t, \tau_1, \tau_2}(\Phi) $ is finite for some  $h>0$.

\par

Conversely, assume that $ q_{\alpha, h}^{s, t, \tau_1, \tau_2}(\Phi) < \infty $ for some  $h>0$.
Employing the same estimate as above, we conclude that
$$
(1 + a^{l_1}) |\partial_a^{\alpha} \partial_b^{\beta} \Phi(b,a)| \lesssim h^{|\beta| + l_1}
\beta!^s l_1!^{\tau_1},
$$
$$ \left(1 + \frac{1}{a^{l_2}}\right)
|\partial_a^{\alpha} \partial_b^{\beta} \Phi(b,a)| \lesssim h^{|\beta| + l_2} \beta!^s
l_2!^{\tau_2}
$$
and
$$ \langle b \rangle^k | \partial_a^{\alpha} \partial_b^{\beta} \Phi(b,a)|
\lesssim h^{|\beta| + k} \beta!^s k!^t,
$$
for every $  ((b, a), (k, l_1, l_2), \beta) \in \Lambda$.
Hence,
$$
\left(a^{l_1} + \frac{1}{a^{l_2}}\right) \langle b \rangle^{k}
|\partial_a^{\alpha} \partial_b^{\beta} \Phi(b,a)|^3 \lesssim h^{3|\beta| + l_1 + l_2 + k}
\beta!^{3s} l_1!^{\tau_1} l_2!^{\tau_2} k!^t,
$$
i.e.
$$
\frac{\tilde{h}^{|\beta| + l_1 + l_2 + k}}{  \beta!^s l_1!^{\tau_1} l_2!^{\tau_2} k!^t}
\left(a^{l_1} + \frac{1}{a^{l_2}}\right) \langle b \rangle^{k}
|\partial_a^{\alpha} \partial_b^{\beta} \Phi(b,a)|
< C
$$
for some $ \tilde{h},C >0. $ By taking the supremum over  $ \Lambda $, we obtain that
$ p_{\alpha, h}^{s, t, \tau_1, \tau_2} (\Phi) $ is finite for some $h>0$.

\end{proof}

\par

%%%%%%%%%%%%%%%%%%%%%

\section{Wavelet transform of ultradifferentiable functions and ultradistributions} \label{Sec2}

In this section we study continuity properties of wavelet
transforms on Gelfand-Shilov spaces of ultradifferentiable functions and their duals. In particular, we derive the
resolution of identity formula in a class of tempered
ultradistributions. As mentioned in the introduction, the most technical proofs
are postponed to Section \ref{Sec3}.

\subsection{Continuity theorems}
A function $ \psi \in {\mathcal S}^{\rho_1}_{\rho_2} (\mathbb{R}^n)$
is called a {\em wavelet} if  $\mu_{0}(\psi ) = 0$.
The wavelet transform of a tempered ultradistribution $ f \in
({\mathcal S}^{\rho_1}_{\rho_2} (\mathbb{R}^n))' $ with respect to the
wavelet $ \psi \in {\mathcal S}^{\rho_1}_{\rho_2} (\mathbb{R}^n)$ is
defined via
\begin{equation}  \label{Wavelet}
\mathcal{W}_{\psi} f (b, a) = \left\langle f (x), \frac{1}{a^n}
\bar{\psi}\left(\frac{x - b}{a}\right) \right\rangle =
\frac{1}{a^n}\int_{\mathbb{R}^n} f (x)  \bar{\psi}\left(\frac{x -
b}{a}\right) \,\mathrm{d}x,
\end{equation}
where $(b,a)\in \mathbb{H}^{n+1}$. In fact, if $ \psi $ is a test function and
the dual pairing in \eqref{Wavelet} makes sense, then we call $ \mathcal{W}_{\psi} f $
the wavelet transform of $f$ with respect to $ \psi $.

\par

We first investigate continuity properties of the
wavelet transform when the analyzing function belongs to a space of
ultradifferentiable functions.

\par

\begin{theorem} \label{th3} Let  $\rho_1 > 0 $, $ \rho_2 > 1 $ and let
$ s >0$, $ t > \rho_1 + \rho_2 $,  $ \tau_1 >\rho_1  $ and $ \tau_2  > \rho_2 - 1$.
Then the mapping
$$
{\mathcal W} : ({\mathcal S}^{\rho_1} _ {\rho_2})_0
(\mathbb{R}^n) \times ({\mathcal S}^{\min\{s, \tau_2 - \rho_2 +1\}}_{1
- \rho_1+  \min\{t - \rho_2, \tau_1\} })_0 (\mathbb{R}^n) \to
{\mathcal S}^{s}_{t, \tau_1, \tau_2} (\mathbb{H}^{n+1}),
$$
given by $ {\mathcal W} : (\psi, \varphi) \mapsto {\mathcal
W}_{\psi} \varphi $, is continuous.
\end{theorem}

\par

\begin{remark}\label{rk2} In the sequel we will use the continuity of
$$
{\mathcal W} : ({\mathcal S}^{\rho_1} _ {\rho_2})_0
(\mathbb{R}^n) \times ({\mathcal S}^{s}_{t + 1
- \rho_1 - \rho_2})_0 (\mathbb{R}^n) \to
{\mathcal S}^{s}_{t, t- \rho_2, s +
 \rho_2  -1} (\mathbb{H}^{n+1})
$$
which follows from Theorem \ref{th3}, when
$ \tau_1 = t- \rho_2 $ and $ \tau_2 = s + \rho_2  -1 $.
\end{remark}

\par

The following simple facts are useful in calculations:

By the Plancherel theorem, we have
$$
{\mathcal W}_{\psi} \varphi (b, a) = \frac{1}{(2\pi)^n} \int_{\mathbb{R}^{n}} e^{ib\xi
}\overline{\hat{\psi}} (a\xi) \hat \varphi (\xi) d\xi, \;\;\; (b, a)
\in  \mathbb{H}^{n+1}.
$$
Hence $ \mathcal{F}_1 {\mathcal
W}_{\psi} \varphi (\xi, a) = \overline{\hat{\psi}} (a\xi) \hat
\varphi (\xi)$, $  (\xi, a) \in  \mathbb{H}^{n+1}$.
Moreover, for $  (b, a) \in  \mathbb{H}^{n+1}, $
$$
\partial_b ^\beta W_{\psi} \varphi (b,a) = \int_{\mathbb{R}^{n}} \varphi ^{(\beta)} (ax + b) \overline{\psi} (x) dx
= i^{|\beta|}
\int_{\mathbb{R}^{n}} e^{i\xi b} \xi^{\beta} \hat{\varphi} (\xi)\overline{\hat{\psi}} (a\xi) d\xi,
$$
and, if  $ \psi \in ({\mathcal S}^{\rho_1} _ { \rho_2})_0
(\mathbb{R}^n) $ then  $ \int b^\gamma {\mathcal W}_{\psi}
\varphi (b, a)  db = 0$, $ \gamma \in \mathbb{N}^n$.

\par

In order to construct the left-inverse for the wavelet transform, we
proceed as follows. The wavelet synthesis transform of $ \Phi \in
{\mathcal S}^{s}_{t, \tau_1, \tau_2} (\mathbb{H}^{n+1}) $, $ s, t,
\tau_1, \tau_2 > 0 $, $ s+t \geq 1$, with respect to $\phi  \in
({\mathcal S}^{\rho_1} _ {\rho_2})_0 (\mathbb{R}^n) $,  $ \rho_1>0$,
$ \rho_2 >1$, is defined by
$$
\mathcal{M}_{\phi} \Phi(x) = \int^{\infty}_{0} \left(
\int_{\mathbb{R}^{n}} \Phi(b,a) \frac{1}{a^{n}} \phi
\left(\frac{x-b}{a} \right)\mathrm{d}b \right) \frac{
\mathrm{d}a}{a}\: , \ \ \ x \in \mathbb{R}^{n}.
$$

\par

\begin{theorem} \label{th4} Let  $ \rho_1 > 0,$ $ \rho_2 > 1 $ and let
$s > 0$, $ t > \rho_2 $ and $ \tau > 0 $. Then the bilinear mappings
\begin{itemize}
\item[a)]$ \displaystyle {\mathcal M} : ({\mathcal S}^{\rho_1} _ {\rho_2})_0
(\mathbb{R}^n) \times {\mathcal S}^{\tau}_{t, t - \rho_2, s - \rho_1} (\mathbb{H}^{n+1})
\to ({\mathcal S}^{s}_{t})_0
(\mathbb{R}^n), $ when $ s > \rho_1$;
\item[b)]$ \displaystyle
 {\mathcal M} : ({\mathcal S}^{\rho_1} _ {\rho_2})_0
(\mathbb{R}^n) \times {\mathcal S}^{s}_{t,t - \rho_2, \tau}
(\mathbb{H}^{n+1}) \to ({\mathcal S}^{s}_{t})_0 (\mathbb{R}^n) $,
\end{itemize}
given by $ {\mathcal M}: (\phi, \Phi) \mapsto {\mathcal M}_{\phi} \Phi $,
are continuous.
\end{theorem}

\begin{remark} \label{remark C and S}
{\em 1.} It will be seen from the proof of Theorem \ref{th4} that a more general statement holds true. In fact, ${\mathcal M}$
can actually be extended to a continuous mapping from
$ {\mathcal S}^{\rho_1} _ {\rho_2}
(\mathbb{R}^n) \times {\mathcal S}^{\tau}_{t, t - \rho_2, s - \rho_1} (\mathbb{H}^{n+1})$
or from
$ {\mathcal S}^{\rho_1} _ {\rho_2} (\mathbb{R}^n) \times {\mathcal S}^{s}_{t,t - \rho_2, \tau}
(\mathbb{H}^{n+1}) $ to $ {\mathcal S}^{s}_{t} (\mathbb{R}^n) $.
However, we will only use wavelets with all vanishing moments in the rest of this article.

\par

{\em 2.}
The continuity properties from Theorem \ref{th4} a) and  b) provide
information about high regularity and the decay properties of $ \mathcal{M}_{\phi}
\Phi $. In the notation of Gelfand-Shilov spaces the
upper index is related to Gevrey ultra-differentiability while the lower index is
related to the decay of a function. Note that when $s=1$, the function $ \mathcal{M}_{\phi}
\Phi $ is real analytic and if $0<s<1$, it extends to an entire function on $\mathbb{C}^{n}$. The index $t$ gives subexponential decay at rate $e^{-h|x|^{1/t}}$, for some $h>0$. In Theorem \ref{th4} a) the
regularity of the image $ \mathcal{M}_{\phi} \Phi $ is measured in
terms of the regularity of the wavelet $\phi $ and the decay of $ \Phi $ when $a>0 $ tends to zero, while Theorem \ref{th4} b) shows that  the regularity of $ \Phi $ is
preserved under the action of the synthesis operator. Similarly, the
decay of $ \mathcal{M}_{\phi} \Phi $ at infinity is related to the
corresponding decays of $\phi $ and $\Phi $ in both Theorem
\ref{th4} a) and b).
\end{remark}

\par

The importance of the wavelet synthesis operator follows from the
fact that it can be used to construct a left inverse for the wavelet
transform, whenever the wavelet possesses nice reconstruction
properties. We end this subsection with a necessary and sufficient condition for such
property to hold in the context of Gelfand-Shilov spaces.

We start with some terminology. We say that a wavelet $ \psi \in {\mathcal S}_{0}
(\mathbb{R}^n)$ admits a reconstruction wavelet $ \phi \in
{\mathcal S}_0 (\mathbb{R}^n) $ if
$$ c_{\psi,\phi}(\omega)= \int^{\infty}_{0}
\overline{\hat{\psi}}(r\omega) \hat{\phi}(r\omega)
\frac{\mathrm{d}r}{r}, \;\;
\omega\in\mathbb{S}^{n-1}, $$
is finite, non-zero, and independent of the direction  $\omega\in\mathbb{S}^{n-1}$. In such a case we write $c_{\psi,\phi}:=c_{\psi,\phi}(\omega)$.

\par

For example, if $ \psi \in \mathcal S_0 (\mathbb{R}^n) $ is
non-trivial and rotation invariant, then it is its own
reconstruction wavelet. In fact, the existence of  a reconstruction wavelet is equivalent to
\emph{non-degenerateness} in the sense of the following definition (see \cite[Proposition 5.1]{vindas-pilipovic-2}).

\begin{definition} \label{non-degenerate}
(\cite{vindas-pilipovic,vindas-pilipovic-2}) A test function $\varphi\in\mathcal{S}
(\mathbb{R}^n)$ is said to be \emph{non-degenerate} if for any
$\omega\in \mathbb{S}^{n-1}$ the function
$R_{\omega}(r)=\hat{\varphi}(r\omega), $ $ r \in [0,\infty)$ is not
identically zero, that is, $ \operatorname*{supp}
R_{\omega}\neq\emptyset, $ for each $ \omega\in\mathbb{S}^{n-1}. $
If in addition $ \mu_0 (\varphi) = 0, $ then $\varphi$ is called a
non-degenerate wavelet.
\end{definition}

\par

We can now state the reconstruction formula for the wavelet transform (cf. \cite[Theorem 14.0.2]{hol1}). If $\psi\in\mathcal{S}_{0}(\mathbb{R}^{n})$ is non-degenerate and $\phi\in\mathcal{S}_{0}(\mathbb{R}^{n})$ is a reconstruction wavelet for it, then
\begin{equation}
\label{reconstruction}
\varphi= \frac{1}{c_{\psi,\phi}}  \mathcal{M}_{\phi}
\mathcal{W}_{\psi} \varphi, \: \: \forall \varphi \in  {\mathcal S}(\mathbb{R}^n).
\end{equation}

We are interested in wavelets in Gelfand-Shilov spaces.
The ensuing proposition shows that if the non-degenerate wavelet $\psi$ possesses higher regularity properties, then it is possible to choose a reconstruction wavelet with the same regularity as $\psi$.

\par

\begin{proposition} \label{lemma1}
Let $ \psi \in ({\mathcal S}^{\rho_1} _ {\rho_2})_0
(\mathbb{R}^n), $ $ \rho_1 > 0,$ $ \rho_2 > 1$, be non-degenerate. Then, it admits a
reconstruction wavelet $ \phi$ such that $\phi\in ({\mathcal S}^{\tau} _ {\rho_2})_0 (\mathbb{R}^n)$, $\forall \tau>0$.
\end{proposition}

\begin{proof}
The proof is similar to that of \cite[Proposition 5.1]{vindas-pilipovic-2}.
However, here the reconstruction wavelet should satisfy additional regularity properties.

Since $\psi$ is non-degenerate, then, by Definition \ref{non-degenerate}, there exist
$0<r_{1}<r_{2}$ such that $ \operatorname*{supp} \, \hat{\psi}(r\omega)
\cap [r_{1},r_{2}]\neq\emptyset$, $\forall \omega\in \mathbb{S}^{n-1}$.

The condition $ \rho_2 > 1$ implies that the space $ {\mathcal S}_{\tau} ^ {\rho_2}(\mathbb{R})$ is non-quasianalytic
for every $ \tau > 0 $.
Let $\eta\in  \bigcap_{\tau>0} {\mathcal S}_{\tau} ^ {\rho_2}(\mathbb{R})$
be a compactly supported nonnegative rotation-invariant function with $0\notin \operatorname*{supp} \eta$
and $\eta(\xi)=1$ for $r_{1}\leq |\xi|\leq r_{2}$.

Consider the auxiliary function
$$
g(\omega)=\int_{0}^{\infty}\eta(r)|\hat{\psi}(r\omega)|^{2}\frac{\mathrm{d}r}{r} > 0,
\;\;\;  \omega\in \mathbb{S}^{n-1}.
$$
A straightforward computation shows that
$$
\sup_{\omega\in \mathbb{S}^{n-1} } | \partial ^{\alpha} g (\omega) | \leq C h^{|\alpha|} (\alpha !)^{\rho_2},
\;\;\; \alpha \in \mathbb{N}^n.
$$
In fact, $g\in \mathcal{E}^{\{(\alpha!)^{\rho_{2}}\}}(\mathbb{S}^{n-1})$,
the Gevrey class of $\{(\alpha!)^{\rho_{2}}\}$-ultra-differentiable functions
on the unit sphere $\mathbb{S}^{n-1}$, see \cite{K} for the definition.
Then, by employing an atlas on $\mathbb{S}^{n-1} $ consisting of functions from
$ \mathcal{E}^{\{(\alpha!)^{\rho_{2}}\}}(\mathbb{S}^{n-1})$ and
\cite[Lemma 1]{K2}, once concludes that
$ 1/g \in \mathcal{E}^{\{(\alpha!)^{\rho_{2}}\}}(\mathbb{S}^{n-1})$, i.e.,
the partial derivatives of $1/g $ satisfy the same decay properties as those of $g$.
Moreover, by \cite[Theorem 8.2.4]{Horm} and since the function $ \omega: \xi \mapsto \xi/|\xi| $ is analytic
off the origin, it follows that $1/g(\xi/|\xi|) \in \mathcal{E}^{\{(\alpha!)^{\rho_{2}}\}} $ away from
the origin.

Finally we define the reconstruction wavelet via its Fourier transform as follows.
Set $\hat{\phi}(\xi):= \eta (\xi)\hat{\psi}(\xi)/g(\xi/|\xi|)$.
It is a compactly supported function, all of its partial derivatives vanish at the origin
and $\hat{\phi}\in \cap_{\tau>0} {\mathcal S}_{\tau} ^ {\rho_2}(\mathbb{R}^n)$.
Therefore $\phi\in ({\mathcal S}^{\tau} _ {\rho_2})_0 (\mathbb{R}^n)$, $\forall \tau>0$.
Moreover, by construction $c_{\psi,\phi}=c_{\psi,\phi}(\omega)=1$. This completes the proof.

\end{proof}

\par

Next we give an example of a non-degenerate wavelet from $(\mathcal{S}^{\rho_{1}}_{\rho_2})_{0}(\mathbb{R}^{n})$.

\begin{example} Assume that $ \rho_1 > 0$ and $ \rho_2 > 1$.
Let $e_j = (0,0,\dots,1,\dots,0), $ with $1$ at the $j-$th
coordinate, and let $ B_{\pm j} = B (\pm
\frac{1}{2} e_j, \frac{1}{2}), $ $ j = 1,2,\dots,n$ denote the closed
balls centered at $ \pm \frac{1}{2} e_j$ with radius $\frac{1}{2} $. Since the class $\mathcal{S}^{\rho_{2}}_{\rho_1}(\mathbb{R}^{n})$ is non-quasiananalytic, it contains compactly supported functions. Set $ \hat \psi=\underset{j\neq 0}{\sum_{j=-n}^{n}  \hat
\phi_{j}}$, where the $ \hat \phi_{\pm j}\in \mathcal{S}^{\rho_{2}}_{\rho_1}(\mathbb{R}^{n})$ are
functions supported by $B_{\pm j} $, $ j = 1,2,\dots,n$,
respectively, and positive in its interior. Then the function $\psi$, the inverse Fourier transform of $\hat \psi$, is an example of
non-trivial non-degenerate wavelet from $ ({\mathcal S}^{\rho_1}_{\rho_2})_0
(\mathbb{R}^n)$.
\end{example}

\par

\subsection{The wavelet transform of tempered ultradistributions}
We start with a useful growth estimate for the wavelet
transform of an ultradistribution. Recall, the wavelet transform of
an  ultradistribution $f$ with respect to the test function $ \psi $ is
given by \eqref{Wavelet} whenever the dual pairing  is well defined.

\par

\begin{proposition} \label{Calderon 2}
Let $ {\rho_1}, {\rho_2} >0, $ $ {\rho_1} + {\rho_2} \geq 1$, $ s> {\rho_1}$ and $ t>{\rho_2}$.
If $ \psi \in {\mathcal S}^{{\rho_1}}_{{\rho_2}} (\mathbb{R}^n) $ and $ f \in
({\mathcal S}^s_t (\mathbb{R}^n))', $ then for  every $ k > 0$,
$$
|{\mathcal W}_{\psi} f(b, a)| \lesssim
 e^{k \big( a^{\frac{1}{t - {\rho_2}}} + (\frac{1}{a})^{\frac{1}{s - {\rho_1}}}
+ |b|^{\frac{1}{t}}  \big)}, \;\;\; (b,a) \in \mathbb{H}^{n+1}.
$$

\end{proposition}

\begin{proof}
Since $ \psi \in {\mathcal S}^{{\rho_1}}_{{\rho_2}} (\mathbb{R}^n)$,
$$
| \partial^\beta \psi (x) | \lesssim \tilde h^{-|\beta|}( \beta ! )^{\rho_1}
e^{-A |x|^{1/\rho_2}}, \;\;\; x \in \mathbb{R}^n, \beta\in \mathbb{N}^n,
$$
for some $\tilde{h},A > 0$. For every $ h > 0 $ there exists $ C_h > 0 $ such
that:
$$
|{\mathcal W}_{\psi} f(b, a)| \leq C_h \, p^{s, t}_h
\left(\frac{1}{a^n} \bar{\psi} \Big(\frac{\: \cdot \: - b}{a}\Big)\right),
\;\;\; (b,a) \in \mathbb{H}^{n+1}.
$$
So we have (for every $c>0$):
$$
|{\mathcal W}_{\psi} f(b, a)| \lesssim \frac{1}{a^n} \sup_{x \in
\mathbb{R}^n, \beta \in \mathbb{N}^n} \frac{h^{|\beta|}}{\beta!^s}
e^{c |ax + b|^{\frac{1}{t}}} \frac{1}{a^{|\beta|}}
|\psi^{(\beta)} (x)|
$$
$$
= \frac{1}{a^n} \sup_{x \in \mathbb{R}^n, \beta \in \mathbb{N}^n}
\frac{\tilde h^{|\beta|} |\psi^{(\beta)} (x)|}{\beta!^{{\rho_1}}}
e^{A |x|^{\frac{1}{{\rho_2}}}} \Big(\frac{h}{\tilde h a}\Big)^{|\beta|}
\frac{1}{\beta!^{s - {\rho_1}}} e^{c |ax + b|^{\frac{1}{t}} - A
|x|^{\frac{1}{{\rho_2}}}},
$$
where $\tilde h, A > 0$. Therefore,
$$
|{\mathcal W}_{\psi} f(b, a)| \lesssim
 p^{{\rho_1}, {\rho_2}}_{l} (\psi) \,
 e^{(s-{\rho_1}) (\frac{h}{\tilde h a})^{\frac{1}{s - {\rho_1} }}} \,
 e^{c |b|^{\frac{1}{t}}}
\, \sup_{x \in \mathbb{R}^n} (e^{c |ax|^{\frac{1}{t}} - A
|x|^{\frac{1}{{\rho_2}}}}),
$$
for some $l>0$. Since $ g(r) := c ( a \cdot r)^{\frac{1}{t}} - A r^{\frac{1}{{\rho_2}}},
\, r > 0 $ attains its  maximal value at $ (\frac{{\rho_2} c}{t
A})^{\frac{t {\rho_2}}{t - {\rho_2}}} a^{\frac{{\rho_2}}{t - {\rho_2}}} $, and since we
may chose arbitrary $h>0$ and $c>0$, it follows that
$$
|{\mathcal W}_{\psi} f(b, a)| \lesssim e^{k
\big((\frac{1}{a})^{\frac{1}{s - {\rho_1} }} + a^{\frac{1}{t - {\rho_2}}} +
|b|^{\frac{1}{t}}  \big)},
$$
for every $k>0$.
\end{proof}

\par

\begin{remark}
\label{remark bounded}
Naturally, if $\psi\in (\mathcal{S}_{\rho_2}^{\rho_1})_{0}(\mathbb{R}^{n})$, then Proposition \ref{Calderon 2}
remains valid for  $f\in ((\mathcal{S}_{t}^{s})_{0}(\mathbb{R}^{n}))'$. Furthermore,
if $\mathcal{B}'\subset (\mathcal{S}_{t}^{s}(\mathbb{R}^{n}))'$ is a bounded set
(resp. $\mathcal{B}'\subset ((\mathcal{S}_{t}^{s})_{0}(\mathbb{R}^{n}))'$
when $\psi\in (\mathcal{S}_{\rho_2}^{\rho_1})_{0}(\mathbb{R}^{n})$)), then the conclusion of
Proposition \ref{Calderon 2} holds uniformly for $f\in \mathcal{B}'$, as follows from
the Banach-Steinhaus theorem.
\end{remark}

Next, we give an alternative definition of the wavelet transform of an ultradistribution via duality.

\par

\begin{definition} \label{def1}
Let $\rho_{1}>0$, $ t > \rho_2>1 $, $s>0$ and $\tau>0$.
If $ \psi \in ({\mathcal S}^{\rho_1}_{\rho_2})_0 (\mathbb{R}^n) $
and $ f \in (({\mathcal S}^s_{t} )_0 (\mathbb{R}^n))'$
then the wavelet transform $ {\mathcal W}_{\psi} f$  of $ f $ with respect to
the wavelet $ \psi$ is defined as
\begin{equation} \label{eq35}
\langle {\mathcal W}_{\psi} f(b, a), \Phi(b, a) \rangle := \langle
f(x), {\mathcal M}_{\bar{\psi}} \Phi (x) \rangle, \: \: \Phi \in {\mathcal S}^s_{t, t - \rho_2, \tau}(\mathbb{H}^{n+1}).
\end{equation}
Thus,
$  {\mathcal W}_{\psi} : (({\mathcal S}^s_{t})_0 (\mathbb{R}^n))'
\to ({\mathcal S}^s_{t, t - \rho_2, \tau}(\mathbb{H}^{n+1}))'$ is continuous for the strong dual topologies.
\end{definition}

\par

By Theorem \ref{th4}  b), the transposition in \eqref{eq35} is well defined. Note that we have freedom of
the choice of $\tau$. This fact will be crucial below. If we assume that $s>\rho_1$, then the choice $\tau=s-\rho_1$ leads to the continuous mapping $  {\mathcal W}_{\psi} : (({\mathcal S}^s_{t})_0 (\mathbb{R}^n))'
\to ({\mathcal S}^s_{t, t - \rho_2, s-\rho_1}(\mathbb{H}^{n+1}))'$.
The next result shows the consistency between Definition \ref{def1} and (\ref{Wavelet}) for this choice of $\tau$.

\begin{proposition}
\label{proposition consistency} Assume that $s > \rho_1>0$ and $ t > \rho_2>1 $. Let $ f \in (({\mathcal S}^s_{t} )_0 (\mathbb{R}^n))'$ and $ \psi \in ({\mathcal S}^{\rho_1}_{\rho_2})_0 (\mathbb{R}^n) $. Then, for every $\Phi \in {\mathcal S}^s_{t, t - \rho_2, s-\rho_1}(\mathbb{H}^{n+1})$,
\begin{equation}
\label{equation consistency}
\langle
f(x), {\mathcal M}_{\bar{\psi}} \Phi (x) \rangle= \int_{0}^{\infty}\int_{\mathbb{R}^{n}}\mathcal{W}_{\psi}f(b,a)\Phi(b,a)\:\frac{dbda}{a}.
\end{equation}
\end{proposition}

\begin{proof}
Fix $\Phi \in {\mathcal S}^s_{t, t - \rho_2, s-\rho_1}(\mathbb{H}^{n+1})$. Proposition \ref{prop2} implies that there is $h>0$ such that
\begin{equation}
\label{eqextra1}
|\Phi(b,a)|\lesssim e^{-h\left(a^{\frac{1}{t-\rho_{2}}} +
a^{-\frac{1}{s-\rho_1}} + |b|^{\frac{1}{t}} \right)}, \;\;\; (b,a) \in \mathbb{H}^{n+1}.
\end{equation}
Let $\{f_{j}\}_{j=0}^{\infty}$ be a sequence such that $f_j\to f$ in $(({\mathcal S}^s_{t} )_0 (\mathbb{R}^n))'$ and $f_{j}\in \mathcal{S}_{0}(\mathbb{R}^{n})$, for every $j\in\mathbb{N}$. In view of Proposition \ref{Calderon 2} (cf. Remark \ref{remark bounded}),
\begin{equation}
\label{eqextra2}
|{\mathcal W}_{\psi} f_{j}(b, a)| \lesssim
 e^{\frac{h}{2} \big(  a^{\frac{1}{t - {\rho_2}}} + (\frac{1}{a})^{\frac{1}{s - {\rho_1}}}
+ |b|^{\frac{1}{t}}  \big)}, \;\;\; (b,a) \in \mathbb{H}^{n+1},
\end{equation}
uniformly in $j\in\mathbb{N}$. Fubini's theorem and the regularity of $f_{j}$ imply
\begin{equation}
\label{eqextra3}
\int_{\mathbb{R}^{n}}f_{j}(x) {\mathcal M}_{\bar{\psi}} \Phi (x) dx= \int_{0}^{\infty}\int_{\mathbb{R}^{n}}\mathcal{W}_{\psi}f_{j}(b,a)\Phi(b,a)\:\frac{dbda}{a}
\end{equation}
Noticing that $\mathcal{W}_{\psi}f_{j}(b,a)\to\mathcal{W}_{\psi} f(b,a) $ pointwisely,
the estimates (\ref{eqextra1}) and (\ref{eqextra2})
allow us to use the Lebesgue dominated convergence theorem in (\ref{eqextra3})
to conclude (\ref{equation consistency}).
\end{proof}

\par

We also introduce the the wavelet synthesis transform of an ultradistribution on $\mathbb{H}^{n+1}$ via a duality approach. The consistency of the following definition is ensured by Theorem \ref{th3} (cf. Remark \ref{rk2}).

\begin{definition} \label{def2}
Let  $ \rho_1 > 0,$ $ \rho_2 > 1 $, $s > 0$ and $ t > \rho_1 + \rho_2 $.
Let $ F \in ({\mathcal S}^s _{t, t - \rho_2, s+\rho_2 -1 }
(\mathbb{H}^{n+1}))' $
and $ \phi \in (\mathcal{S} ^{\rho_1 } _{\rho_2})_0 (\mathbb{R}^n) $.
The wavelet
synthesis transform $ {\mathcal M}_{\phi} F$ of $ F $ with respect
to the wavelet $ \phi $  is defined by
$$
\langle {\mathcal M}_{\phi} F(x), \varphi (x) \rangle := \langle
F(b, a), {\mathcal W}_{\overline{\phi}} \varphi (b, a) \rangle,
\;\;\; \varphi \in ({\mathcal S}^s _{t+1 - \rho_1 -\rho_2})_0 (\mathbb{R}^n).
$$
Thus, $ {\mathcal M}_{\phi} : ({\mathcal S}^s _{t, t - \rho_2, s+\rho_2 -1 }
(\mathbb{H}^{n+1}))' \to (({\mathcal S}^s _{t+1 - \rho_1 -\rho_2})_0 (\mathbb{R}^n))'$ is continuous.
\end{definition}

\par

We derive the following resolution of the identity mapping
$ \mathrm{Id}$ as an easy consequence of our previous results. In the next theorem we implicitly use the choice $\tau=s+\rho_{2}-1$ in Definition \ref{def1}.

\par

\begin{theorem} \label{th-res-identity}
Let  $ \rho_1 > 0,$ $ \rho_2 > 1 $, $s > 0$ and $ t > \rho_1 + \rho_2 $.
Let $\psi \in  (\mathcal{S}^{\rho_1}_{\rho_2})_0 (\mathbb{R}^n) $
be a non-degenerate wavelet and let $ \phi \in  (\mathcal{S}^{\rho_1}_{\rho_2})_0 (\mathbb{R}^n) $
be a reconstruction wavelet for it. Then the Calder\'{o}n reproducing formula
$$
\mathrm{Id} = \frac{1}{c_{\psi,\phi}} \mathcal{M}_{\phi}
\mathcal{W}_{\psi}
$$
holds in $ (({\mathcal S}^ s_{t})_0 (\mathbb{R}^n))' $.
\end{theorem}

\begin{proof}
Let $ f \in (({\mathcal S}^s_{t})_0 (\mathbb{R}^n))'. $
Since $ ({\mathcal S}^s _{t+1 - \rho_1 -\rho_2})_0 (\mathbb{R}^n) $
is dense in the space $ ({\mathcal S}^s_{t})_0 (\mathbb{R}^n) $,
it is enough to prove the identity for test functions
$ \varphi \in ({\mathcal S}^s _{t+1 - \rho_1 -\rho_2})_0 (\mathbb{R}^n).$
Then, by Definitions \ref{def2} and \ref{def1}, and the reconstruction formula \eqref{reconstruction}, it follows that
\begin{equation} \label{jednakosti}
\langle  \mathcal{M}_{\phi} \mathcal{W}_{\psi} f,  \varphi \rangle
=
\langle  \mathcal{W}_{\psi} f, \mathcal{W}_{\bar{\phi}} \varphi
\rangle = \langle f, \mathcal{M}_{\overline{\psi}}
\mathcal{W}_{\overline{\phi}} \varphi \rangle =
c_{\bar{\phi},\bar{\psi}} \langle f, \varphi \rangle.
\end{equation}
\end{proof}

\par

Combining Remark \ref{rk2}, Proposition \ref{proposition consistency}, and the relation (\ref{jednakosti}), we obtain an extension of the {\em desingularization} formula now in the context of ultradistributions
(cf. \cite{hol1,vindas-pilipovic-2} for the case of distributions).

\par

\begin{corollary}
\label{desingularization ultra} In addition to the assumptions of Theorem \ref{th-res-identity} suppose that $\sigma:=\rho_{1}+\rho_{2}-1<s$. Then,
$$
\langle f, \varphi \rangle = \frac{1}{c_{\psi,\phi}}
\int_0 ^\infty \int_{\mathbb{R}^n}
\mathcal{W}_\psi f (b,a)
\mathcal{W}_{\overline{\phi}} \varphi (b,a)  \frac{dbda}{a}, \: \: \forall \varphi\in (\mathcal{S}^{s-\sigma}_{t-\sigma})_{0}(\mathbb{R}^{n}).
$$
\end{corollary}

\par

As an immediate consequence of Theorem \ref{th-res-identity} and
Theorem \ref{th4} b) we have the following regularity theorem for ultradistributions.

\begin{corollary} Let  $ \rho_1 > 0,$ $ \rho_2 > 1 $, $s > 0$ and $ t > \rho_1 + \rho_2 $.
Let $\psi \in  (\mathcal{S}_{\rho_2}^{\rho_1})_0 (\mathbb{R}^n) $ be non-degenerate
and let $f \in (({\mathcal S}^s _{t})_0 (\mathbb{R}^n))'$.
If $ \mathcal{W}_{\psi} f \in
\mathcal{ S}^s _{t, t-\rho_2, \tau} (\mathbb{H}^{n+1})$ for some $
\tau > 0$ then $ f \in ({\mathcal S}^s_{t})_{0} (\mathbb{R}^n) $.
\end{corollary}

\section{Proofs of main results} \label{Sec3}

This section collects the proofs of  Theorems \ref{th3} and
\ref{th4}.

\begin{remark} \label{enlarge h}
On several occasions we will use the following facts. If $ \varphi
\in \mathcal{S}_{{\rho_2}} ^{{\rho_1}} (\mathbb{R}^n) $, $ {\rho_1} + {\rho_2} \geq 1$,
so that $ p^{{\rho_1}, {\rho_2}} _{h_0} (\varphi) < \infty$ for some $h_0 >0$,
then there exists $h_1>0$ such that
\begin{equation} \label{integral}
\sup_{\alpha, \beta \in \mathbb{N} ^n} \frac{ h_1 ^{|\alpha +
\beta|} }{ \alpha !^{\rho_2} \beta !^{\rho_1}} \int | x |^{|\alpha|} |\varphi
^{(\beta)} (x)| dx \lesssim p^{{\rho_1}, {\rho_2}} _{h_0} (\varphi).
\end{equation}
In addition, if  \eqref{integral} holds, then there exists $h_2>0$ such
that
$$
\sup_{\alpha, \beta \in \mathbb{N} ^n} \frac{ h_2 ^{|\alpha +
\beta|} }{ \alpha!^{\rho_2} \beta!^{\rho_1}}
 | x |^{|\alpha|} | \varphi ^{(\beta)} (x) | < \infty.
$$
We will omit the parts of the proofs where these arguments
appear. We will often make use of the fact that multiplication by $
|\cdot |^{|\alpha|} $ (or by $\langle \cdot \rangle^{|\alpha|} $)
simply enlarges the corresponding constants $h>0$.
We will also use, without explicit reference, the estimate
$$
\sup_{\beta \in \mathbb{N}^n, x \in \mathbb{R}^n} \frac{h_0
^{|\beta|} (fg)^{(\beta + \alpha)}(x)}{\beta!^s } \lesssim
\sup_{\beta\in \mathbb{N}^n, x \in \mathbb{R}^n } \frac{h_1
^{|\beta|} f^{(\beta)}(x)}{\beta!^s } \sup_{\beta\in \mathbb{N}^n, x
\in \mathbb{R}^n } \frac{h_1 ^{|\beta|} g^{(\beta)} (x)}{\beta!^s }
$$
for some $ h_1 = h_1 (\alpha) > 0 $. Finally, we shall need the following form of the reminder term in the Taylor formula
$$
(R_{\alpha,m} f)(x,y) =  \sum_{|\alpha| = m}
\frac{m(y-x)^{\alpha}}{\alpha!} \int_0 ^1 (1-\theta)^{m-1}
 f^{(\alpha)} (x + \theta(y-x))d\theta.
$$
\end{remark}

\subsection{Proof of Theorem \ref{th3}} Let
$ \varphi $ and $\psi $ satisfy
\begin{equation} \label{fi-psi-norme}
p^{\min\{s, \tau_2 - \rho_2+1 \}, 1 - \rho_1 + \min\{t  - \rho_2, \tau_1 \}} _{h } (\varphi)
p^{\rho_1, \rho_2}_{h} (\psi) < \infty.
\end{equation}
We will show that there exists $h_0>0$ such that $ p_{\alpha, h_0}^{s,
t, \tau_1, \tau_2} (\mathcal{W}_\psi \varphi) <\infty$, that is, the supremum
of
$$
J = \frac{h_0 ^{|\beta| +k+ l_1 + l_2} (a^{l_1} + a^{-l_2}) \langle
b \rangle^k |  \partial_a ^{\alpha} \partial_b ^{\beta} \mathcal{W}_\psi
\varphi (b,a) |}{ \beta!^s k!^{t} l_1 !^{\tau_1}  l_2 !^{\tau_2}}
$$
over
$$ ((b, a), (k, l_1, l_2), \beta) \in \Lambda = \mathbb{H}^{n+1} \times \mathbb{N}^3 \times
\mathbb{N}^n
$$
is finite for some $h_0>0$. Without loss of generality we assume from now on that $ \alpha = 0$.

\par

\begin{remark} \label{h constants}
In the following steps of the proof we start with $h_0>0$ and in
every step determine a new constant $h_1 \leq  h_2 \leq \dots \leq
h_7$ which successively depends on the previous one, that is, $ h_m $ depends on
$ h_{m-1}$, $ m = 1,2,\dots,7$ and $h_7$ should be equal to  $h>0$ in
\eqref{fi-psi-norme}.
Then, by going in the opposite direction, we determine   $h_6 $ from $h_7$,
$h_5 $ from $ h_6 $, $ \dots ,$ and $ h_0$ from $h_1$. In such way for the constant $h>0$ given in
\eqref{fi-psi-norme} we find $h_0 > 0 $ so that
$$
p^{s, t, \tau_1, \tau_2} _{0,h_0} (\mathcal{W}_\psi \varphi) = \sup_{\Lambda}
J \lesssim p_{h}^{\min\{s, \tau_2 - \rho_2+1\}, 1 - \rho_1 + \min\{t  -
\rho_2, \tau_1 \}} \, (\varphi) p^{\rho_1, \rho_2}
_{h} (\psi),
$$
which will prove the Theorem.
\end{remark}

\par

We will estimate
$$
J_1 = \frac{h_0 ^{|\beta| +2k+ 2l_1} (1 + a^{2l_1}) \langle b
\rangle^{2k} | \partial_b ^{\beta} \mathcal{W}_\psi \varphi (b,a) |}{ \beta!^s
k!^{2t} l_1 !^{2\tau_1}}
$$
over $ ((b, a), (k, l_1), \beta) \in \Lambda_1 = \mathbb{H}^{n+1}
\times \mathbb{N}^2 \times \mathbb{N}^n,$ and
$$
J_2 = \frac{h_0 ^{|\beta| + 2l_2}  (1 + a^{-2l_2}) | \partial_b
^{\beta} \mathcal{W}_\psi \varphi (b,a) |}{ \beta!^s  l_2 !^{2\tau_2}}
$$
over $ ((b, a), (k, l_2), \beta) \in \Lambda_2 = \mathbb{H}^{n+1}
\times \mathbb{N}^2 \times \mathbb{N}^n$.

\par

Since
$$
\frac{h_0 ^{2(|\beta| +k+ l_1 + l_2)} (a^{l_1} + a^{-l_2})^{2}\langle b
\rangle^{2k} | \partial_b ^{\beta} \mathcal{W}_\psi \varphi (b,a) |^2}{ \beta!^{2s}
k!^{2t} l_1 !^{2\tau_1}  l_2 !^{2\tau_2}} \lesssim \sup_{\Lambda_1}
J_1  \sup_{\Lambda_2} J_2
$$
we would have
$$
p^{s,t, \tau_1, \tau_2} _{0,h_0} (\mathcal{W}_\psi \varphi) \lesssim
\sqrt{\sup_{\Lambda_1} J_1 \sup_{\Lambda_2} J_2}.
$$
We will show that there exists $h_7 >0$ which depends on $h_0>0$
such that
$$
\sup_{\Lambda_1} J_1 \lesssim p_{h_7}^{s, 1 - \rho_1 + \min\{t  - \rho_2,
\tau_1 \}} (\varphi) \, p^{\rho_1, \rho_2}_{h_7}
(\psi)
$$
and
$$
\sup_{\Lambda_2} J_2 \lesssim p_{h_7}^{\min\{s, \tau_2 - \rho_2 +1\},t}
(\varphi) \, p^{\rho_1, \rho_2} _{h_7} (\psi).
$$

\par

We first estimate $\sup_{\Lambda_1} J_1$. There exists $ h_1 = h_1
(h_0) $  such that
$$
J_1 \lesssim \frac{h_1 ^{|\beta| +2k+ 2l_1} (1 + a^{2l_1}) \langle b
\rangle^{2k}}{ \beta!^s (2k)!^{t} (2l_1) !^{\tau_1} \langle b
\rangle^{2k}} \left| \int_{\mathbb{R}^{n}} e^{i\xi b} (1-\triangle_\xi)^k (\xi^{\beta}
\hat{\varphi} (\xi) \overline{\hat{\psi}} (a\xi)) d\xi \right|
$$
$$
= \frac{h_1 ^{|\beta| +2k+ 2l_1} }{ \beta!^s (2k)!^{t} (2l_1)
!^{\tau_1}} \left| \int_{\mathbb{R}^{n}} e^{i\xi b} \sum_{|\gamma| \leq 2k} c_\gamma
\partial ^\gamma _\xi (\xi^{\beta} \hat{\varphi} (\xi) (1+ a^{2l_1})
\overline{\hat{\psi}} (a\xi)) d\xi \right|
$$

$$
\lesssim \frac{h_1 ^{|\beta| +2k+ 2l_1} }{ \beta!^s (2k)!^{t} (2l_1)
!^{\tau_1}}
 \sum_{|\gamma| \leq 2k} |c_\gamma | \sum_{i+j \leq \gamma} |\tilde c_{i,j}|
\int_{\mathbb{R}^{n}}   |\partial ^i (\xi^{\beta} \hat{\varphi} (\xi))|   a^{|j|}
(1+a^{2l_1})|\hat{\psi} ^{(j)} (a\xi)| d\xi,
$$
where $ c_\gamma $ and $ \tilde c_{i,j} $ are correspondent binomial
coefficients. As already noticed, by the use of Leibniz rule, the
binomial coefficients simply increase the constant $h_1$ so that
$$
J_1 \lesssim \sup \frac{h_2 ^{|\beta| +2k+ 2l_1} }{ \beta!^s
(2k)!^{t} (2l_1) !^{\tau_1}} \sum_{|i+j|\leq 2k}(I_1 + I_2),
$$
for some  $ h_2 = h_2 (h_1) >0$ which does not depend on $\beta, k$
and $l_1$, where
$$
I_1 = \int_{|\xi| \leq 1}    |\partial ^i (\xi^{\beta} \hat{\varphi}
(\xi))|   a^{|j|}(1+a^{ 2l_1})  |\hat{\psi} ^{(j)} (a\xi)| d\xi,
$$
and
$$
I_2 = \int_{|\xi| \geq 1}   |\partial ^i (\xi^{\beta} \hat{\varphi}
(\xi))|   a^{|j|}(1+a^{ 2l_1})  |\hat{\psi} ^{(j)} (a\xi)| d\xi,
$$
and the supremum is taken over $\beta, k$ and $l_1$.
\par

By Remarks \ref{remark equivalent sequences} and \ref{enlarge h} it
follows that there exists $ h_3 = h_3(h_2) > 0$, which does not
depend on $\beta, i,j$ and $l_1$,  such that
$$
\frac{h_2 ^{|\beta| + |i|+|j| + 2l_1} }{ \beta!^s |i|!^{t} |j|!^{t}
(2l_1) !^{\tau_1}} I_2
$$
$$
\lesssim \frac{h_2 ^{|\beta| + |i|+|j| + 2l_1} }{ \beta!^s |i|!^{t}
|j|!^{t} (2l_1) !^{\tau_1}}
 \int_{|\xi| \geq 1}   | \partial ^i
(\xi^{\beta} \hat{\varphi} (\xi))| |\xi|^{n} a^{n}  |a\xi|^{|j|-n}(1+|a\xi|^{2l_1}) |
\hat{\psi} ^{(j)} (a\xi)| d\xi
$$
$$
\lesssim \sup_{} \frac{h_3 ^{|\beta| + |i|} }{ \beta!^s |i|!^{t} }\langle \xi \rangle ^{n} |
\partial ^i (\xi^{\beta} \hat{\varphi} (\xi))   |\sup_{} \frac{h_3
^{|j| + 2l_1} }{  |j|!^{t} (2l_1) !^{\tau_1}} \int_{\mathbb{R}^{n}}
|x|^{|j|-n} (1+|x|^{2l_1})| \hat{\psi} ^{(j)} (x)| dx
$$
$$
\lesssim p_{h_3 } ^{s, t} (\varphi) p^{\min\{t  - \rho_2,
\tau_1\},  \rho_2} _{h_3} (\psi),
$$
where the suprema are taken over $\beta, i,j, l_1 $ and $ \xi $, and
we have splited $|j|!^t$ into $|j|!^{t -\rho_2} $ and  $|j|!^{\rho_2} $.

\par
%%%%%%%%%%%%%%%%%%%%%%%%%%%

From the above calculations we conclude that there exists $ h_3> 0$
such that
$$
\sup_{\Lambda_1} \frac{h_2 ^{|\beta| +2k+ 2l_1} }{ \beta!^s
(2k)!^{t} (2l_1) !^{\tau_1}}  I_2 \lesssim p_{h_3 } ^{s, t}
(\varphi) p^{\min\{t  - \rho_2, \tau_1\},  \rho_2} _{h_3}
(\psi).
$$

%%%%%%%%%%%%%%%%%%%%%%%%%%
\par

Next, we estimate the term with $I_1$:
$$
I_1 \lesssim \sum_{p\leq i} {i \choose p} \frac{\beta!}{p!} \left |
\int_{|\xi| \leq 1} \xi^{\beta - p} \hat \varphi ^{(p)} (\xi)
a^{|j|}(1+a^{2l_1}) \overline{\hat{\psi}} ^{(j)} (a\xi) d\xi \right |
$$
$$
= \sum_{p\leq i} {i \choose p} \frac{\beta!}{p!} \Big | \int_{|\xi|
\leq 1} \xi^{\beta - p} \sum_{|r|= 2l_1 + |j|-n} \frac{(2l_1 + |j|)
\xi^{r}}{r!} \cdot I \cdot a^{|j|} (1+a^{2l_1})
\overline{\hat{\psi}}^{(j)}  (a\xi) d\xi \Big |
$$
where $ I = \int_0 ^1 (1-\theta)^{2l_1 + |j|-n-1} \hat \varphi ^{(p +
r)} (\theta \xi) d\theta$, and we have used Taylor's formula for $ \hat
\varphi $ and the vanishing moments of $ \varphi$.

\par

Since $ \displaystyle \sum_{|r|= 2l_1 + |j|-n} \frac{1 }{r!} \lesssim \frac{c ^{2l_1
+ |j|}}{(2l_1)!|j|!} $ for some $c>0$ and binomial coefficients just
increase the constant $h_2$, we obtain
$$
\frac{h_2 ^{|\beta| +2k+ 2l_1} }{ \beta!^s (2k)!^{t} (2l_1)
!^{\tau_1}} I_1
$$
$$
\lesssim \sup \frac{ h_4 ^{|\beta| +|i|+|j| + 2l_1} }{ \beta!^s
|i|!^{t} |j|!^{t+1} (2l_1) !^{\tau_1+1}}|\hat \varphi ^{(i+ r)} (x) |\int_{|\xi|\leq1}
a^{n}(1+|a\xi|^{|j|+2l_1-n}) |\hat{\psi} ^{(j)}  (a\xi)| d\xi
$$
$$
\lesssim\sup h_4 ^{|\beta| +|i|+|j| + 2l_1} \frac{
|\hat \varphi ^{(i+ r)} (x) |}{ \beta!^s |i|!^{t} |j|!^{t-
\rho_1 -\rho_2 +1} (2l_1) !^{\tau_1+1-\rho_1}}  \int_{\mathbb{R}^{n}} \frac{\langle \xi\rangle^{|j|+
2l_1-n} |{\hat\psi} ^{(j)}  (\xi)|}{|j|!^{\rho_1 }
(2l_1) !^{\rho_1}|j|!^{\rho_2 } } d\xi
$$
with $|r|= 2l_1 + |j|$, the suprema taken over $\beta, i, j,
l_1$ and $x$, and where $h_4  >0 $ does not depend on $\beta, i,
j, l_1$. Moreover, we may choose $ h_4 \geq
h_3$.

\par

By Remarks \ref{remark equivalent sequences} and \ref{enlarge h} and
similar arguments to those used in the estimates of $I_2$ it follows
that there exists $h_5 = h_5 (h_4)>0$ (which does not depend on
$\beta, k$ and $ l_1$) such that
$$
\frac{h_2 ^{|\beta| +2k+ 2l_1} }{ \beta!^s (2k)!^{t} (2l_1)
!^{\tau_1}} I_1 \lesssim p^{s, 1- \rho_1+ \min\{t  - \rho_2, \tau_1 \}} _{h_5} (\varphi) \,
p^{\rho_1, \rho_2} _{h_5}(\psi).
$$

Since the sequence $h_0,  h_1,\dots, h_5 $ is non-decreasing, we
conclude that
$$
\sup_{\Lambda_1} J_1 \lesssim p^{s, 1- \rho_1+  \min\{t  - \rho_2,
\tau_1 \}}_{h_5} (\varphi) \, p^{\rho_1, \rho_2}_{h_5} (\psi).
$$

\par

It remains to estimate $ \sup_{\Lambda_2} J_2 $. We now use
Taylor's formula for $ \varphi $ and the vanishing moments of $ \psi $
to obtain, for $a<1$,
$$
J_2 = \frac{h_0 ^{|\beta| + 2l_2}  (1 + a^{-2l_2}) }{ \beta!^s
(2 l_2)!^{\tau_2}} \times
$$
$$
\times \left|\int_{\mathbb{R}^{n}} \left(  \sum_{|r| = 2l_2} \frac{2l_2}{r!} \left(\int_0
^1 (1-\theta)^{2l_2-1} \varphi^{(\beta + r)} (b+ \theta a x )
d\theta \right) a^{2l_2} x^r \overline{\psi} (x) \right) dx\right|
$$
$$
\lesssim \frac{h_6 ^{|\beta| + 2l_2} }{ \beta!^s (2l_2) !^{\tau_2 +
1}} \sup_{x \in \mathbb{R}^n} \max_{|r|=2l_{2}} |\varphi^{(\beta + r)} (x) |
\sup_{x \in \mathbb{R}^n} | \langle x \rangle ^{r+n+1} \overline{\psi}  (x) |
$$
$$
\lesssim p^{\min\{s, \tau_2 - \rho_2 +1\}, t} _{h_7} (\varphi) \,
p^{\rho_1,  \rho_2}_{h_7} (\psi),
$$
for some $h_6 \geq h_5 $ and $ h_7 = h_7 (h_6)$. When $a\geq1$, we employ a similar argument (Taylor's formula is not needed for this case).

\par

Thus, we choose $h_7 = h >0 $ for which
$$
p^{\min\{s, \tau_2 - \rho_2 +1\}, 1- \rho_1+ \min\{t  - \rho_2, \tau_1 \}} _{h } (\varphi)
\, p^{\rho_1, \rho_2}_{h} (\psi) < \infty.
$$
Now, reasoning as in Remark \ref{h constants}, we determine for given $h_7 = h$ the corresponding
$ h_0 > 0 $  so that
$$
p^{s, t, \tau_1, \tau_2} _{0, h_0} (\mathcal{W}_\psi \varphi)\lesssim
p^{\min\{s, \tau_2 - \rho_2 +1\}, 1- \rho_1+ \min\{t  - \rho_2, \tau_1 \}} _{h_7 } (\varphi) \, p^{\rho_1,  \rho_2}_{h_7} (\psi)
< \infty,
$$
which proves the Theorem.

\par

\subsection{Proof of Theorem \ref{th4}} We may again assume that $ \alpha = 0$.

a) Let $h>0$ be chosen so that
$$
p^{\tau,t, t - \rho_2, s - \rho_1}_{0, h} (
\Phi) \, p^{\rho_1, \rho_2}_{h} (\phi) < \infty.
$$
By Lemma \ref{FTH} it is enough to prove that there exists $h_0>0$
such that
\begin{equation} \label{eq9}
p^{s,t}_{h_0} ({\mathcal M}_{\phi} \Phi) \lesssim p^{t, \tau,
t - \rho_2, s - \rho_1}_{0, h} ({\mathcal F}_1 \Phi) \,
p^{\rho_1, \rho_2}_{ h} (\phi).
\end{equation}

\par

Let $  \beta \in \mathbb{N}^n$ and $k \in \mathbb{N} $.  We may assume that $k$ is even. Then
$$
\langle x \rangle^{k} \left|\partial_x^{\beta} \big({\mathcal M}_{\phi}
\Phi (x)\big)\right| = \langle x \rangle^{k} \left|\partial_x^{\beta}
\int_{\mathbb{R}_+} \int_{\mathbb{R}^n} \Phi (b, a) \frac{1}{a^n}
\phi \left(\frac{x - b}{a}\right) \, db \, \frac{da}{a}\right|
$$
$$
\lesssim \langle x \rangle^{k} \left|\int_{\mathbb{R}_+}
\int_{\mathbb{R}^n}
\partial_b^{\beta} \Phi (x - b, a) \frac{1}{a^n} \phi \left(\frac{b}{a}\right) \, db \, \frac{da}{a}\right|
$$
$$
\lesssim \langle x \rangle^{k} \left|\int_{\mathbb{R}_+}
\int_{\mathbb{R}^n}   e^{-i x \xi} \xi^{\beta} \hat{\Phi} (-\xi, a)
\hat{\phi} (a \xi) \, d\xi \, \frac{da}{a}\right|
$$
$$
\lesssim \langle x \rangle^{k} \left|\int_{\mathbb{R}_+}
\int_{\mathbb{R}^n} \frac{(1 - \Delta_{\xi})^{\frac{k}{2}} e^{-i x
\xi}}{ \langle x \rangle^{k} } \xi^{\beta} \hat{\Phi} (-\xi, a)
\hat{\phi} (a \xi) \, d\xi \, \frac{da}{a}\right|
$$

$$
\lesssim \sum_{|r| + |q| \leq k}  \int_{\mathbb{R}_+}
\int_{\mathbb{R}^n}
\partial_{\xi}^r (\xi^{\beta} \hat{\Phi} (-\xi, a))
\partial_{\xi}^{q} (\hat{\phi} (a \xi))| \, d\xi \, \frac{da}{a}
\lesssim  \sum_{|r| + |q| \leq k} \, I,
$$
where
$$
I = \int_{\mathbb{R}_+} \int_{\mathbb{R}^n} a^{|q|} \,
|\xi|^{|\beta|} \, | \partial_{\xi}^r \hat{\Phi} (-\xi, a)| |\hat{\phi}^{(q)}
(a \xi)| \, d\xi\, \frac{da}{a}.
$$

\par

We use again Remark \ref{enlarge h}
in a similar way as it was done in the proof of Theorem \ref{th3}.
In the corresponding steps of the proof we enlarge $h_0 > 0 $ and
regroup the integrands in an appropriate way. In fact, by taking the
corresponding suprema, one can show that there exist $h_1 = h_1 (h_0)
>0$ and $h_2 = h_2 (h_1) >0$, which do not depend on $\beta, q$ and $r$,
such that
\begin{equation} \label{eq10}
\frac{h_0 ^{|\beta| + k}}{\beta!^{s} k!^{t}}
\langle x \rangle^{k} |\partial_x^{\beta} \big({\mathcal M}_{\phi}
\Phi (x)\big)| \lesssim
\frac{h_0 ^{|\beta| + k}}{\beta!^{s} k!^{t}} \sum_{|r| + |q|
\leq k} I
\end{equation}
$$
\lesssim \sup  h_1 ^{|\beta| + |r| + |q|} \frac{\left ( a^{|q|} +
\frac{1}{a^{|\beta| +1}} \right) |\partial_{\xi}^r \hat{\Phi} (-\xi,
a)|}{\beta!^{s -\rho_1} q!^{t - \rho_2} r!^{t}}
\frac{|a\xi|^{|\beta|} |\hat{\phi}^{(q)} (a \xi)|}{\beta!^{\rho_1}
q!^{ \rho_2}}
$$
$$
\lesssim \, p^{t, \tau, t - \rho_2, s - \rho_1}_{0, h_2}
({\mathcal F}_1 \Phi) \, p^{\rho_1, \rho_2}_{h_2} (\phi),
$$
where the supremum is taken over $ \xi, a, \beta, r$ and $q$ (we
have also used $ a^{|q|-|\beta|-1} \leq \left ( a^{|q|} +
\frac{1}{a^{|\beta| +1}} \right),$ $a>0$).

By Remark \ref{h constants} for $h_2 = h$, there exists $h_0 > 0 $
such that
$$
\sup_{\beta \in  \mathbb{N}^n, k \in  \mathbb{N}} \frac{h_0
^{|\beta| + k}}{\beta!^{s} k!^{t}}
 \sum_{|r| + |q| \leq k} I \, \lesssim \,
p^{t, \tau, t - \rho_2, s - \rho_1}_{0, h} ({\mathcal F}_1
\Phi) \, p^{\rho_1, \rho_2}_{h} (\phi) ,
$$
which implies \eqref{eq9}.

\par

b) Here we bound (\ref{eq10}) by
$$
\sup  h_3 ^{|\beta| + |r| + |q|} \frac{\left ( a^{|q|} + \frac{1}{a}
\right) |\xi|^{|\beta|} |\partial_{\xi}^r \hat{\Phi} (-\xi,
a)|}{\beta!^{s} q!^{t - \rho_2} r!^{t}}
\frac{|\hat{\phi}^{(q)} (a\xi)|}{ q!^{ \rho_2}}
$$
$$
\lesssim \, p^{t, s, t - \rho_2, \tau}_{0, h_4} ({\mathcal
F}_1 \Phi) \, p^{\rho_1, \rho_2}_{ h_4} (\phi),
$$
for some $ h_3 = h_3 (h_0),$ and $ h_4 = h_4 (h_3),$ where the
supremum is taken over $ \xi, a, \beta, r$ and $q$. Once again by Remark
\ref{h constants} for $h_4 = h$, it follows that there exists $h_0 > 0
$ such that
$$
\sup_{\beta \in  \mathbb{N}^n, k \in  \mathbb{N}} \frac{h_0
^{|\beta| + k}}{\beta!^{s} k!^{t}} \sum_{|r| + |q| \leq k}  I \,
\lesssim \, p^{t, s, t - \rho_2, \tau}_{0, h} ({\mathcal
F}_1 \Phi) \, p^{\rho_1, \rho_2}_{ h} (\phi) < \infty,
$$
which completes the proof.

\section*{Acknowledgement}

S. Pilipovi\'{c} and N. Teofanov are supported by the Ministry of
Education, Science and Technological Development of the Republic of
Serbia through Project 174024. D. Raki\'{c} is supported by the
Ministry of Education, Science and Technological Development of the
Republic of Serbia through Project III44006 and by PSNTR through
Project 114-451-2167. J. Vindas acknowledges support by Ghent University, through the BOF-grant 01N01014.

\end{document}